\documentclass[11pt]{amsart}

\usepackage{amsmath,amssymb,amsbsy,amsfonts,amsthm,latexsym,amsopn,amstext,amsxtra,epic, euscript,amscd,indentfirst, tikz}
\usepackage{verbatim} 
\usepackage{hyperref} 
\usepackage[margin=3.5cm]{geometry}
\usepackage{pgfplots}
\usepackage{graphicx}
\usepgfplotslibrary{patchplots}
\usetikzlibrary{patterns, positioning, arrows}
\usetikzlibrary{arrows,positioning} 
\pgfdeclarepatternformonly{my crosshatch dots}{\pgfqpoint{-1pt}{-1pt}}{\pgfqpoint{5pt}{5pt}}{\pgfqpoint{6pt}{6pt}}%
{
    \pgfpathcircle{\pgfqpoint{0pt}{0pt}}{.5pt}
    \pgfpathcircle{\pgfqpoint{3pt}{3pt}}{.5pt}
    \pgfusepath{fill}
}

\DeclareFontFamily{U}{mathc}{}
\DeclareFontShape{U}{mathc}{m}{it}%
{<->s*[1.03] mathc10}{}

\DeclareMathAlphabet{\mathscr}{U}{mathc}{m}{it}

\begin{document}

\newtheorem{theorem}{Theorem}
\newtheorem*{theorem*}{Theorem}
\newtheorem{conjecture}[theorem]{Conjecture}
\newtheorem{proposition}[theorem]{Proposition}
\newtheorem{question}[theorem]{Question}
\newtheorem{lemma}[theorem]{Lemma}
\newtheorem{definition}[theorem]{Definition}
\newtheorem*{definition*}{Definition}
\newtheorem{cor}[theorem]{Corollary}
\newtheorem*{cor*}{Corollary}
\newtheorem*{result*}{Result}
\newtheorem{obs}[theorem]{Observation}
\newtheorem{proc}[theorem]{Procedure}
\newcommand{\comments}[1]{} 
\def\Z{\mathbb Z}
\def\Za{\mathbb Z^\ast}
\def\Fq{{\mathbb F}_q}
\def\R{\mathbb R}
\def\N{\mathbb N}
\def\i{\sqrt{-1}}
\def\k{\kappa}

\title{On the total curvature of confined equilateral quadrilaterals}

 \author[University of Michigan]{Gabriel Khan}

\email{gabekhan@umich.edu}

\date{\today}

\maketitle 

\begin{abstract}
In this paper, we prove that the total expected curvature for random spatial equilateral quadrilaterals with diameter at most $r$ decreases as $r$ increases. To do so, we prove several curvature monotonicity inequalities and stochastic ordering lemmas in terms the of the action-angle coordinates. Using these, we can use Baddeley's extension of Crofton's differential equation to show that the derivative of the expected total curvature is non-positive.
\end{abstract}



\section{Introduction}

Random spatial polygons have been studied extensively from many different viewpoints. The original motivation for this problem is that spatial polygons are a simple model for the folding of polymers. As such, it is of considerable interest to understand the statistical properties of the geometry and topology of such objects.

The moduli space of spatial polygons with given edge lengths is a fascinating object in its own right. When $n =4$ or a generic assumption on the edge lengths is made, this moduli space is a smooth K\"ahler manifold of dimension $2n - 6$ \cite{KM}. There are still many open questions about this space, which remains an active subject of research. It is also possible to use the symplectic structure of the moduli space to establish theorems about the geometry of random walks and polygons \cite{CS}.

In this work, we study confined random polygons, using the diameter of a polygon as a measure of its confinement. There are other possible measures as well, such as the confinement radius or the gyration radius of the polygon \cite{KCM}. Imposing confinement on a random polygon affects the geometry in subtle ways and is a topic of active research. 
In this work we study the curvature of confined random polygons, which was previously studied by Diao et al. \cite{DEMZ}. We define $M_{n,r}$ to be the moduli space of equilateral spatial polygons with diameter at most $r$ and $ \bar \kappa_{n,r}$  the expected total curvature of polygons in $M_{n,r}$. More precisely,  
\[ \bar \kappa_{n,r} = \frac{1 }{\mu(M_{n,r})}  \int_{P_n \in M_{n,r}} \sum_{i=1}^n \angle(e_i,e_{i+1}) \, d \mu,\]
where $d \mu$ the symplectic volume form.

It has been observed that confinement tends to increase the expected total curvature. Numerical simulation bears this out and heuristically, a polygon must turn on itself repeatedly in order to remain confined. However, proving this rigorously for general polygons seems to be a difficult problem. The main contribution of our work is to prove this conjecture for equilateral random quadrilaterals.

\begin{theorem}
$\bar \kappa_{4,r}$ is non-increasing as $r$ increases.
\end{theorem}

This result gives some insight, albeit indirectly, into the probability that a random problem is knotted. This is a question of considerable interest, as the knottedness of a polymer can directly affect its physical properties. Heuristically, we expect that confinement forces the polygon to tangle and so confined polygons are more likely to be knotted than unconfined ones. This has been demonstrated numerically \cite{MMOS}, but it is very difficult to explicitly compute knotting probabilities.

 Polygons with curvature less than $4 \pi$ are unknots via the Fary-Milnor theorem \cite{JM}, so expected total curvature can be used as a very rough proxy for knottedness. As all quadrilaterals and pentagons are unknotted, our main theorem does not give direct insight into the knottedness phenomena. However, we also show that when the number of sides $n$ is even, the probability of a random polygon being knotted is decreasing when the confinement diameter is close to $n/2$.

\subsection{Acknowledgments}

Thanks to Clayton Shonkwiler for informing the author of this problem and for some helpful discussions. Thanks also to Alex Wright for providing some flexible polygons made from straw and string which were very helpful for experiments. He also provided some very helpful lectures and notes \cite{AW} on the moduli space of spatial polygons. This work was partially supported by DARPA/ARO Grant W911NF-16-1-0383 (Information Geometry: Geometrization of Science of Information, PI: Zhang).

\section{Notation and Conventions} \label{Notation}

 To define the moduli space of equilateral spatial polygons, we consider $\mathbb{S}_1^2 \subset \mathbb{R}^3$ the unit sphere and consider the map $\mathcal{D}$:

\begin{eqnarray*} \mathcal{D}: & \underbrace{\mathbb{S}^2_1 \times \ldots \times  \mathbb{S}^2_1}_{n \textrm{-times}} &  \to \mathbb{R}^3 \\
 & (e_1, \ldots , e_n) & \mapsto \sum_{i=1}^n e_i 
\end{eqnarray*}

The moduli space of spatial equilateral polygons $M_n$ is defined to be $\mathcal{D}^{-1}(0) / SO(3)$. Colloquially, $\mathcal{D}^{-1}(0)$ are the collection of edges that sum to 0 (i.e form a closed polygon), from which we quotient out the action of isometries.  In the language of \cite{AW}, $M_n$ is equivalent to $ \mathcal{M} ( \underbrace{1,\ldots,1}_{n \textrm{-times}} )$ but we will focus on equilateral polygons and so suppress the repeated ones. It can be shown that $M_n$ is \textit{not} a manifold in general ($0$ is not a regular value for $\mathcal{D}$ when $n$ is even), but is a manifold when $n$ is odd or $n$ is equal to 4.

\begin{center}
\begin{tikzpicture}
\begin{axis}[
  view={35}{15},
  axis lines=center,
  width=15cm,height=10cm,
  xtick={0,1,2,3},ytick={-1,0,1},ztick={0,1,2},
  xmin=0,xmax=3.2,ymin=-.7,ymax=2.3,zmin=0,zmax=2,
]


\addplot3 [only marks] coordinates { (0+.2,0+.2,0+.2) (1.414+.2,1.414*.382+.2,1.414*.924+.2) (2.828+.2,0+.2,0+.2) (1.414+.2,-1.414*.382+.2,1.414*.924+.2) };

\addplot3 [thick, black!60!green] coordinates { (0+.2,0+.2,0+.2) (1.414+.2,1.414*.382+.2,1.414*.924+.2) (2.828+.2,0+.2,0+.2) (1.414+.2,-1.414*.382+.2,1.414*.924+.2) (0+.2,0+.2,0+.2) };
\addplot3 [dashed] coordinates { (0+.2,0+.2,0+.2) (2.828+.2,0+.2,0+.2) };

\node [above left] at (axis cs:.2,.2,.2) {$v_1$};
\node [above right] at (axis cs:1.614,.74015,1.5065) {$v_4$};
\node [above right] at (axis cs:3.028,.2,.2) {$v_3$};
\node [above right] at (axis cs:1.614,-0.340,1.5065) {$v_2$};
\node [above] at (axis cs:1.614,.2,.2) {$\ell_3$};
\end{axis}
\end{tikzpicture}

Figure 0: An equilateral spatial quadrilateral with $\ell_3$ labeled.
\end{center}

Given an equilateral spatial polygon $P_n \in M_n$, we can consider its vertices $ \{ v_i \}_{i=1}^n$ and edges $ \{ e_i \}_{i=1}^n$ with the convention that $e_i$ connects $v_i$ and $v_{i+1}$. From this, we induce action-angle coordinates $\{ (\ell_i, \theta_i) \}_{i=3}^{n-1}$ on $M_n$, where, at a point $P_n$, $\ell_i = d(v_i,v_1)$ and the angle coordinate is given by the angle of rotation about the line $\overline{v_1 v_i}$. For quadrilaterals, $\theta_3$ is the angle between the triangles $\triangle(v_1 v_2 v_3)$ and $\triangle(v_1 v_4 v_3)$. The action-angle values form coordinates on all but a set of positive codimension on $M^n$ and induce $M_n$ with a natural symplectic structure with invariant volume form $d \mu = \prod_{i=3}^{n-1} d \ell_i d \theta_i$ \cite{KM}. For most of this paper, we will not use of the symplectic structure, but do use its associated volume form.

In order to consider confined polygons, we denote $M_{n,r}$ to be the moduli space of equilateral polygons with diameter at most $r$. More precisely,
\[M_{n,r} := \left \{ P_n \in M_n ~|~ \| \sum_{i=j}^k e_i \| < r \textrm{ for all }1 \leq  j,k \leq n \right \}. \]

We then define $ \bar \kappa_{n,r}$  the total expected curvature of polygons in $M_{n,r}$:
\[ \bar \kappa_{n,r} = \frac{1 }{\mu(M_{n,r})}  \int_{P_n \in M_{(n,r)}} \sum_{i=1}^n \angle(e_i,e_{i+1}) d \mu. \]

\subsection{ Notation for quadrilaterals}

We now specialize to quadrilaterals. To aid the intuition, it might be worthwhile to observe that $M_4$ is a topologically Riemann sphere but its intrinsic metric may not be round.  Although we will not need to use this fact explicitly, it helps inform the geometric intuition.

 To better understand $M_4$, we introduce a slight change of coordinates that are useful for computation. Given $P_4 \in M_4$, we can perform an isometry on $P_4$ so that 
\begin{eqnarray*}
 v_1=(0,0,0) & v_2 = (\cos \phi, \sin \phi, 0 ) \\
  v_3 = (2 \cos \phi,0, 0 ) & v_4 = (\cos \phi, \sin \phi \cos \theta,\sin \phi \sin \theta).
\end{eqnarray*}
 
We treat $\phi$ and $\theta$ as coordinates, in which case we have $\ell_3= 2 \cos \phi$ and $\theta_3 = \theta$. Since there is a single $\ell_3$ and $\theta_3$ coordinate in this case, we will drop the subscripts, denoting them as $\ell$ and $\theta$, respectively.

 Note that there are symmetries of an equilateral polygon induced by relabeling the vertices, which correspond to distinct points in $M_4$. Using these symmetries will help simplify the calculations. 
For instance, we denote $M_4^+$ to be the subset of $M_4$ where $0 \leq \theta \leq \pi$. Similarly, we denote  $M_{4,r}^+$ to be the subset of  $M_{4,r}$ where $0 \leq \theta \leq \pi$. Any quadrilateral whose vertices are in the above form is either in $M_4^+$ or can be reflected to a polygon in $M_4^+$. As such, polygons in $M_{4,r}$ and $M_{4,r}^+$ have the same expected curvature, so it suffices to work solely in terms of $M_{4,r}^+$.

\section{The geometry of $M_{4,r}^+$ and its boundary sets}

We now define some natural boundary sets of $M_{4,r}$. We define $\partial_\epsilon  M_{4,r}^+$ to be the set  $ M_{4,r+\epsilon}^+ \backslash M_{4,r}^+$. Letting $\epsilon$ go to zero, we define $\partial  M_{4,r}^+$ as the set-theoretic limit 
\[ \partial  M_{4,r}^+ = \lim_{\epsilon \to 0} \partial_\epsilon  M_{4,r}^+. \]

More explicitly, $\partial M_{4,r}^+$ consists of two separate parts; the set where $\ell=r$ and $d(v_2,v_4)\leq r$ and the set where $d(v_2,v_4)=r$ and  $\ell \leq r$. We denote the former part by  $\partial^\ell M_{4,r}^+{}$ and the latter by  $\partial^\theta M_{4,r}^+{}$.  Note that $\partial M_{4,r}^+$ is not the boundary of $M_{4,r}^+$ in the traditional sense, as it does not include the parts where $\ell=0$, $\theta=0$, or $\theta= \pi$. In Figure 1, $M_{4,r}^+$ is the shaded region in $(\theta, \ell)$-coordinates. The top and right parts of the boundary is $\partial M_{4,r}^+$.

\begin{center}
\includegraphics[width=90mm,scale=0.5]{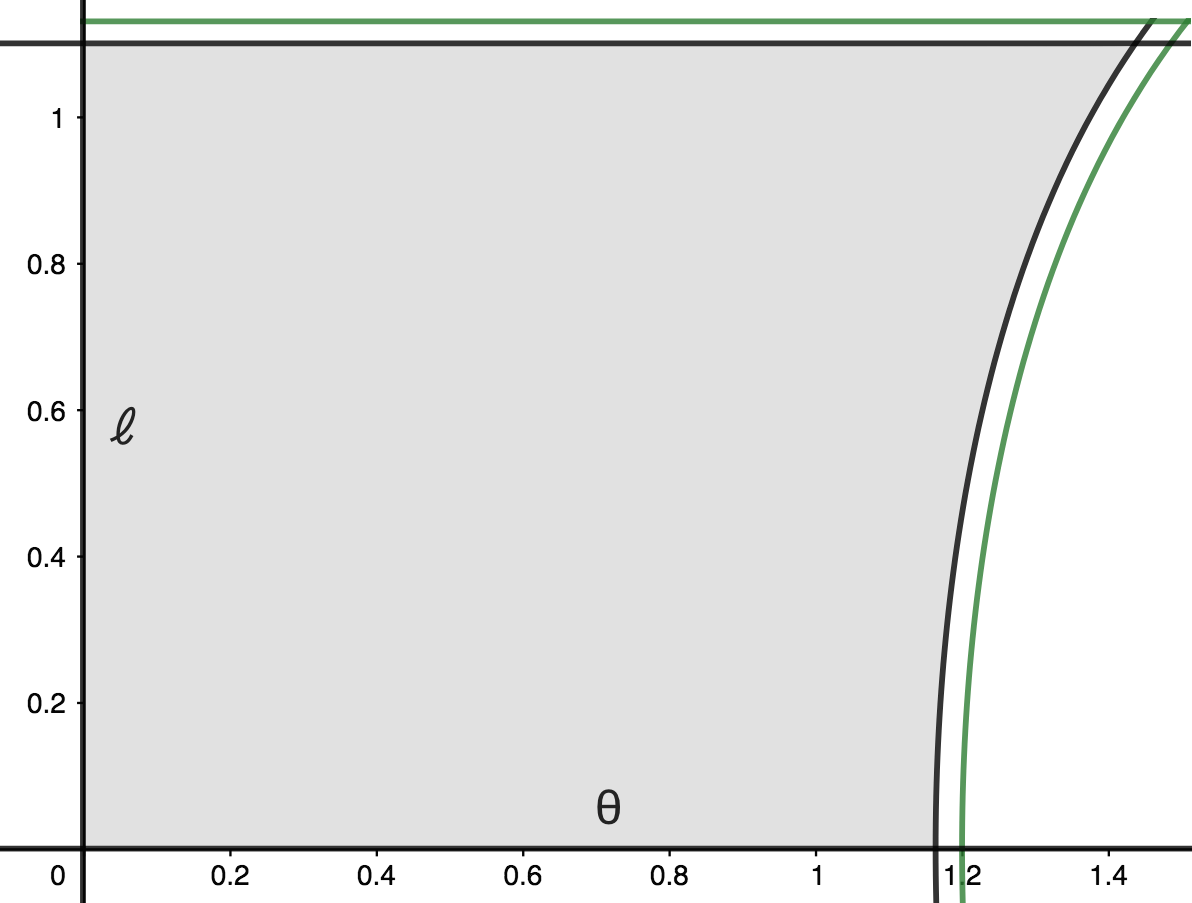}

Figure 1:  $M_{4,r}^+$ and $\partial M_{4,r}^+$ for $r=1.1$
\end{center}

We can define two projections from $M_{4,r}^+$ onto $ \partial  M_{n,r}^+ $. The first of these, $\pi_\ell$ projects $M_{4,r}^+$ to $ \partial ^\ell M_{n,r}^+ $ by fixing the $\theta$ coordinate. The second, $\pi_\theta$ projects $M_{4,r}^+$ to $ \partial ^\theta M_{n,r}^+ $ in a way that fixes the $\ell$ coordinate. More precisely, for a given $P_4(\ell, \theta) \in M_{4,r}^+$,
\begin{eqnarray*} \pi_\ell \left( P_4(\ell, \theta) \right) &=&  P_4(r,\theta) \\
 \pi_\theta \left( P_4( \ell,\theta) \right) &=& P_4 \left( \ell, \arccos \left( 1 - \frac{r^2}{2-\frac{\ell^2}{2}} \right) \right).
\end{eqnarray*}

 Using the uniform measure $\mu$ on $M_{4,r}^+$, we induce $\partial M_{4,r}^+$ with two measures $\mu_B$ and $\mu_I$. Heuristically, $\mu_B$ is the natural boundary measure, whereas $\mu_I$ is the marginal probability measure induced from the uniform measure on $ M_{4,r}^+$. To define $\mu_B$, we consider $ \mu_\epsilon$ the uniform measure on $\partial_\epsilon M_{4,r}^+$, normalized so that its total measure is 1. We define $ \mu_B$ to be the limit of these measures on $\partial  M_{4,r}^+$. This turns out to be a uniform measure on $\partial^\ell M_{4,r}^+$ , but is \textit{not} the uniform measure on $\partial^\theta M_{4,r}^+$. In Figure 1, the distance of the green curve from the black corresponds to the density $d \mu_B$.

Intuitively, we define $\mu_I$ as the marginal distributions of $\ell$ and $\theta$ with respect to the normalized uniform measure on $ M_{4,r}$. For purposes that will later become clear, we want the measure $\mu_I(\partial^\theta M_{4,r}^+)$ to be the same as $\mu_B(\partial^\theta M_{4,r}^+)$. To ensure this, we set $\alpha = \mu_B( \partial^\ell M_{4,r}^+{})$ and  $(1-\alpha) = \mu_B( \partial^\theta M_{4,r}^+{})$. Using symmetry, it is possible to show that $\alpha = \frac{1}{2}$, but we will not prove this here.

 For $U \subset \partial^\ell M_{4,r}^+{}$, we define
 \[ \mu_I(U) = \frac{ \alpha}{ \mu(M_{4,r}^+)}  \mu \left( \left\{ P_4(\ell, \theta) ~|~ \ell<r~ (r, \theta) \in U \right\} \right). \]

Similarly, for $U \subset  \partial^\theta M_{4,r}^+{}$ and $r<\sqrt{2}$, we define
 \[\mu_I(U) =\frac{ 1- \alpha}{ \mu(M_{4,r}^+)} \mu \left( \left\{ P_4(\ell, \theta) \in M_{4,r}^+ ~|~  \exists \, \theta^\prime \textrm{ such that } (\ell, \theta^\prime ) \in U \right\} \right). \] 

When $r > \sqrt{2}$, $\mu_I$ is not a probability measure on $\partial M_{4,r}^+$. The reason for this is that for a polygon $P_4(\ell,\theta)$ with $\sqrt{2} < \ell < r$, $d(v_2, v_4) \leq \sqrt{2}$ independent of $\theta$. To avoid this issue, we restrict our attention to the range $r<\sqrt{2}$. We will return to the case for larger $r$ later.

\subsection{Semi-explicit calculations of $\mu_B$ and $\mu_I$}
It is necessary to calculate these measures more explicitly to understand their properties. To do so, note that
\begin{eqnarray*}
d(v_2,v_4)^2 & = & \left( \sin \phi - \sin \phi \cos \theta \right)^2 +  - \sin^2 \phi \sin^2 \theta \\
& =&  \sin^2 \phi (2 - 2 \cos \theta).
\end{eqnarray*}
In $(\ell, \theta)$ coordinates, this is given by the expression
 \[ d(v_2,v_4)^2 = \left( 1-\frac{\ell^2}{4} \right) (2-2 \cos \theta). \]

For $ 1< r< \sqrt{2}$, this allows us to write out the boundary sets explicitly.

  \[  \partial^\ell M_{4,r}^+ = \left\{ P_4(\ell, \theta) ~|~ \ell =r \textrm{ and }\theta <  \arccos \left( 1 - \frac{r^2}{2-\frac{\ell^2}{2}} \right) \right\} \]
  \[ \partial^\theta M_{4,r}^+ = \left\{ P_4(\ell, \theta) ~|~  \ell < r \textrm{ and }\theta =  \arccos \left( 1 - \frac{r^2}{2-\frac{\ell^2}{2}} \right)  \right\} \]

In order to find the density $d \mu_B$, we must calculate the derivatives $ \frac{\partial \theta }{\partial r} $ and $ \frac{\partial \ell }{\partial r}$. Doing so, we find
  \[  d \mu_B \propto \begin{cases}  \frac{\partial \theta }{\partial r} =  \dfrac{2 r}{ \left(2-l^2/2 \right)   \sqrt{ 1- \left(1-\frac{r^2}{2- \ell^2/2}\right)^2 } } \textrm{ on } \partial^\theta M_{4,r}^+{} \\
    \frac{\partial \ell}{\partial r} \equiv 1  \textrm{ on } \partial^\ell M_{4,r}^+{}. \end{cases} \]

The measure $\mu_I$ satisfies the following.
\[ d \mu_I \propto \arccos \left( 1 - \frac{r^2}{2-\frac{\ell^2}{2}} \right) \textrm{ on } \partial^\theta M_{4,r}^+. \]

 We will not compute $d \mu_I$ explicitly on $\partial^\ell M_{4,r}^+$. However, for a fixed $r$ value, $d \mu_I$ is proportional to the height of the shaded region in Figure 1 as a function of $\theta$.

\section{Stochastic orderings of the boundary measures}

We now prove various lemmata on the stochastic ordering of $\mu_B$ and $\mu_I$. To do so, we first note the following two lemmas. These can be proven directly by differentiating the relevant distance formulas, so we omit the proofs here.

\begin{lemma}
For a fixed action coordinate $\ell$, the distance $d(v_2,v_4)$ is monotonically increasing in $\theta$ and the other distances are unchanged. As such, $M_{4,r}^+$ is star-shaped with respect to $\theta$
\end{lemma}

 Since for fixed $\theta$, $d(v_2,v_4)$ is decreasing in the $\ell$-coordinate, we have the following.

\begin{lemma}
If $\ell < \ell^\prime < r$ and $P_4(\ell, \theta) \in M_{4,r}^+ $, then $P_4(\ell^\prime, \theta) \in M_{4,r}^+ $.  
\end{lemma}

Combining these two lemmas, this implies that the density $d \mu_I$ is non-increasing as a function of $\theta$ on $\partial^\ell M_{4,r}^+$. Since $\mu_I$ and $\mu_B$ are normalized to have the same total mass on $\partial^\ell M_{4,r}^+$, this shows the following.

\begin{lemma} \label{stoch1}
 The measure $\mu_I$ is stochastically less than $\mu_B$ on $\partial^\ell M_{4,r}^+$ as a function of $\theta$.
\end{lemma}

Although it is more difficult to prove, a similar phenomena also occurs on $\partial^\theta M_{4,r}^+$.

\begin{lemma} \label{stoch2}
 For $r< \sqrt{2}$, the measure $\mu_I$ is stochastically less than $\mu_B$ on $\partial^\theta M_{4,r}^+$ as a function of $\ell$.
\end{lemma}

\begin{proof}
To show this, we will prove the monotonicity of likelihood ratio property, which implies first-order stochastic dominance. On $\partial^\theta M_{4,r}^+$, we consider  
\begin{eqnarray*}  \frac{d \mu_B}{d \mu_I} &=& \frac{\partial}{\partial r} \log \left[ \arccos \left( 1 - \frac{r^2}{2-\frac{\ell^2}{2}} \right) \right]   \\
& =& \frac { 2 r } { \left( 2 - \frac { \ell ^ { 2 } } { 2 } \right) \sqrt { 1 - \left( 1 - \frac { r ^ { 2 } } { 2 - \frac { \ell ^ { 2 } } { 2 } } \right) ^ { 2 } } \arccos \left[ 1 - \frac { r ^ { 2 } } { 2 - \frac { \ell ^ { 2 } } { 2 } } \right] }
\end{eqnarray*}

We want to show that this is increasing in $\ell$. To do so, we take a further derivative.

\begin{eqnarray*}
	\frac{\partial}{\partial \ell} \frac{d \mu_B}{d \mu_I} & = & - \frac{2 \ell r ^ { 3 } \left( - 2 \sqrt { - \frac { r ^ { 2 } \left( - 4 +\ell ^ { 2 } + r ^ { 2 } \right) } { \left( - 4 + \ell ^ { 2 } \right) ^ { 2 } } } + \arccos \left[ \frac { - 4 + \ell ^ { 2 } + 2 r ^ { 2 } } { - 4 + \ell ^ { 2 } } \right] \right)}{      \left( - 4 + \ell ^ { 2 } \right) ^ { 3 } \left( - \frac { r ^ { 2 } \left( - 4 +\ell ^ { 2 } + r ^ { 2 } \right) } { \left( - 4 + \ell ^ { 2 } \right) ^ { 2 } } \right) ^ { 3 / 2 } \arccos \left[ \frac { - 4 + \ell ^ { 2 } + 2 r ^ { 2 } } { - 4 + \ell ^ { 2 } } \right] ^ { 2 }}
\end{eqnarray*}

The denominator of this term is negative, so we disregard it and consider only the terms in parenthesis in the numerator, which we define as $\Psi_1$:

\[ \Psi_1(\ell,r) := - 2 \sqrt {  \frac { r ^ { 2 } \left(  4 -\ell ^ { 2 } - r ^ { 2 } \right) } { \left(  4 - \ell ^ { 2 } \right) ^ { 2 } } } + \arccos \left[ \frac { - 4 + \ell ^ { 2 } + 2 r ^ { 2 } } { - 4 + \ell ^ { 2 } } \right] \]

If we can show that $ \Psi_1(\ell,r)$ is non-negative for $\ell,r >0$, then necessarily the entire expression will be as well. However, $ \Psi_1$ vanishes at $\ell=0$. As such, we consider $\frac{\partial \Psi_1}{\partial \ell}$ and show that this is non-negative. Doing so, we find that

\[ \frac{\partial \Psi_1}{\partial \ell} =  \frac { 4 \ell r ^ { 4 } } { \left( 4 -\ell  ^ { 2 } \right) ^ { 3 } \sqrt { - \frac { r ^ { 2 } \left( - 4 + \ell ^ { 2 } + r ^ { 2 } \right) } { \left( - 4 +\ell ^ { 2 } \right) ^ { 2 } } } } \]

Since this is non-negative for $\ell>0$, $\Psi_1$ is non-negative for $\ell >0$ and hence $ \frac{d \mu_B}{d \mu_I}$ is non-decreasing on $\partial^\theta M_{4,r}^+$. This implies that $\mu_I$ is stochastically less than $\mu_B$, as desired. 

\end{proof}

\section{Monotonicity of total curvature}

We now consider the curvature of spatial quadrilaterals, in order to show that larger polygons have smaller total curvature. More precisely, we show that the curvature is decreasing in the the $\ell$ and $\theta$ coordinates. We first show that the curvature is decreasing if one increases the angle coordinate while leaving the $\ell$ coordinate fixed.

\begin{lemma} \label{unfoldmonotone}
Suppose $P_4$ has action-angle coordinates $ (\ell, \theta)$. Then, the $\ell$-coordinates, the total curvature of the spatial polygon is monotonically decreasing in $\theta$ as $\theta$ varies from $0$ to $\pi$.
\end{lemma}

\begin{proof}
Changing $\theta_j$ only changes the angle between $\angle(e_1,e_4)$ and $\angle(e_2,e_3)$. As such, it suffices to show that both of these are decreasing in $\theta_j$.  As before, we suppose $e_1= (\cos \phi , \sin \phi , 0)$, $e_4= (-\cos \phi, -\sin \phi \cos \theta, \sin \theta \sin \phi)$. 
We consider $\phi= \angle e_1,e_4$  as the argument $\angle(e_2,e_{3})$ is exactly the same. The angle $\phi$ satisfies $\cos(\phi)= \langle e_1,e_4 \rangle =  -\cos^2 \phi - \sin^2 \phi \cos \theta$, which is decreasing in $\theta$ as $\theta$ ranges from $0$ to $\pi$. 
\end{proof}

The curvature is also decreasing if we increase the action coordinate. From experimentation with spatial polygons, this is intuitively plausible, but it is not so simple to prove analytically. The reason for this is that when $\theta = \pi$, the total curvature is constant in $\ell$. As such, any proof must be sensitive to the fact that all derivatives of the total curvature vanish when $\theta = \pi$. 

\begin{lemma} \label{stretchmonotone}
For a quadrilateral with fixed angle coordinate $\theta$, the total curvature is non-increasing in $\ell$.
\end{lemma}

\begin{proof}
Using our initial embedding for quadrilaterals, we can see that the total curvature $\kappa$ is
\begin{eqnarray*}
 \kappa &= & 2 \arccos(e_1\cdot e_4)+2 \arccos(e_1 \cdot e_2) \\
            & = & 2 \arccos(- \cos^2 \phi  \sin^2 \phi \cos \theta) + 2 \arccos(\cos^2 \phi- \sin^2 \phi) \\
            & =&  2 \arccos( -\cos^2 \phi - \sin^2 \phi \cos \theta) +4 \phi 
\end{eqnarray*}

To continue, we change our coordinates to $t=\cos^2 \phi$ and $c=\cos \theta$. In these new coordinates,
\begin{eqnarray*}
 \kappa &= &  2 \arccos( -t - (1-t)c) +2 \arccos(2t-1).
\end{eqnarray*}
Taking the derivative of $\kappa$ with respect to $t$, we find the following.
\begin{eqnarray*}
\frac{ \partial \kappa}{\partial t} &= & - \frac { 2 ( - 1 + c ) } { \sqrt { 1 - ( - c ( 1 - t ) - t ) ^ { 2 } } } - \frac { 4 } { \sqrt { 1 -(-1+ 2t)^2 } } \\
\end{eqnarray*}

To show that this expression is non-positive, we fix $t$ and maximize the first term with respect to $c$. Note that this is equivalent to maximizing \[  \frac  { \sqrt { 1 - ( - c ( 1 - t ) - t ) ^ { 2 } } } {  ( - 1 + c ) }, \]
which is the slope of the secant line for the function \[ \Psi_2(x) = \sqrt { 1 - ( - x ( 1 - t ) - t ) ^ { 2 } } \] through the points $x=c$ and $x=1$. Computing the second derivative of $\Psi_2$, we find the following.
\[\frac{d^2 \Psi_2}{d x^2} = - \frac { ( t-1 ) ^ { 2 } } { \left( (t-1)(-1+t(-1+x)-x)(x-1) \right)^{3/2}} \]

This is non-positive and so $\Psi_2$ is concave. Since $-1 \leq c <1 $, in order to maximize the slope of the secant line, we set $c=-1$. Doing so, we find that
\[ \frac{ \partial \kappa}{\partial t} \leq  \frac { 4 } { \sqrt { 1 -(-1+ 2t)^2 } }- \frac { 4 } { \sqrt { 1 -(-1+ 2t)^2 } } = 0 \]
\end{proof}

These two lemmas show that if we increase either the angle or the action (or both), the total curvature decreases. When combined with Lemma 1, the results of this section show the following.

\begin{lemma}
The total curvature of a quadrilateral in $\partial^\ell M_{4,r}^+$ is decreasing as a function of $\theta$. Similarly, the total curvature of a quadrilateral in $\partial^\theta M_{4,r}^+$ is non-increasing as a function of $\ell$.
\end{lemma}

\section{The derivative of the expected total curvature}

In order to calculate the derivative of the expected total curvature, we will use Crofton's differential equation. This formula was first derived by Crofton in 1885 but was only proven rigorously in later work of Baddeley \cite{AB}. For a good survey on the topic, we refer to the paper of Eisenberg and Sullivan \cite{ES}. In this section, we also use the notion of transport for probability measures. For a complete reference on this topic, we refer the reader to the first chapter of the book by Villani \cite{OTON}.

We define $\kappa_{4,r, \mu_B}$ as the expected total curvature when polygons are chosen from $\partial M_{4,r}^+$ with respect to the measure $\mu_B$. More precisely,
\[\kappa_{4,r, \mu_B} := \int_{P_n \in \partial M_{n,r}^+ } \kappa(P_n)\, d \mu_B. \]
We also define $\kappa_{4,r, \mu_I}$ to be the total expected curvature when the polygons are chosen from $\partial M_{4,r}^+$ with respect to the measure $\mu_I$:
\[\kappa_{4,r, \mu_I} := \int_{P_n \in \partial M_{n,r}^+ } \kappa(P_n)\, d \mu_I. \]

With this notation, Crofton's differential equation shows the following:
\[ \frac{ d \bar \kappa_{4,r}}{d r} = \frac{ \left(  \kappa_{4,r,\mu_B} - \bar \kappa_{4,r} \right)}{ \mu(M^+_{4,r})} \frac{d}{dr} \mu(M^+_{4,r}) \]


Combining Lemma 7 and the stochastic ordering lemmas (Lemmas \ref{stoch1} and \ref{stoch2}), this shows that $\kappa_{4,r, \mu_B} < \kappa_{4,r, \mu_I}$. To compare $ \kappa_{4,r, \mu_I} $ and $\bar \kappa_{4,r}$, note that there is a natural transport from $\mu$ and $\mu_I$, induced by $ \alpha (\pi_\ell)_\# \mu + (1- \alpha) (\pi_\theta)_\# \mu$. Lemmas \ref{unfoldmonotone} and \ref{stretchmonotone} imply that this transport decreases the total curvature, which implies $\kappa_{4,r, \mu_I} < \bar \kappa_{4,r}$.

 Combining the previous two inequalities, we find that $\kappa_{4,r,\mu_B} < \bar \kappa_{4,r} $. As such, the second term in the above differential equation is non-positive, so the expected total curvature is non-increasing. This proves the following.

\begin{theorem}
For $1 \leq r \leq \sqrt{2}$,
$\bar \kappa_{4,r}$ is non-increasing in $r$.
\end{theorem}

\section{The proof for $r > \sqrt{2}$}

In this section, we prove that $\bar \kappa_{4,r}$ is also decreasing for $ \sqrt{2} < r < 2$. In this range, the diameter is exactly twice the radius of the polygon so this proves that the expected curvature is also decreasing as a function of the radius. It is worth noting that this approach can be adapted to work for $r< \sqrt{2}$ as well. 

\begin{center}
\includegraphics[width=90mm,scale=0.4]{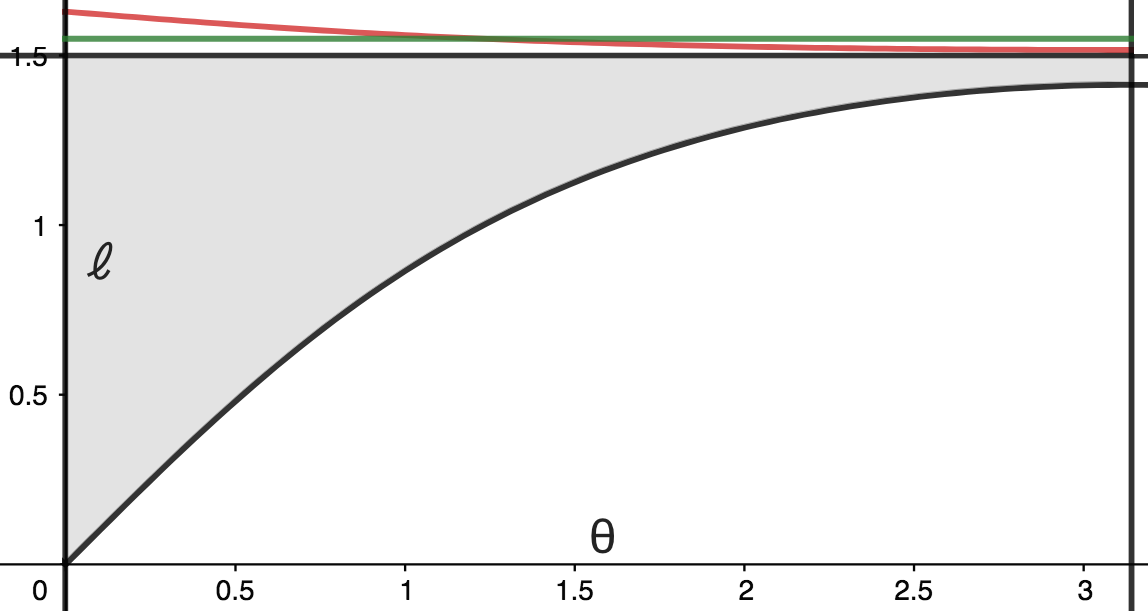}

Figure 2: $M_{4,r}^{+,\ell}$ and its boundary $\partial M_{4,r}^{+,\ell}$ for $r=1.5 $
\end{center}

We define the set $M_{4,r}^{+,\ell}$ to be
 \[ M_{4,r}^{+,\ell} := \left \{ P_n \in M_{4,r}^{+} ~|~ d(v_1,v_3) \geq d(v_2,v_4) \right \}. \] 
Explicitly, this is the subset of $M_{4,r}^{+}$ where $\theta< \arccos\left( \dfrac{4-3\ell^2}{4-\ell^2} \right)$  and is depicted in Figure 2 in $( \theta, \ell)$ coordinates. Given any polygon $P_4 \in M_{4,r}$, we can find an associated polygon in $M_{4,r}^{+,\ell}$ which is obtained from $P_4$ by a mirror reflection and a relabeling of the vertices. As such, the expected curvature on $M_{4,r}^{+,\ell}$ is the same as the expected curvature on $M_{4,r}^{+}$.

 We now define the natural boundary of $M_{4,r}^{+,\ell}$:
\[ \partial M_{4,r}^{+,\ell} := \lim_{\epsilon \to 0} M_{4,r+\epsilon}^{+, \ell} \backslash M_{4,r}^{+, \ell}. \]

 As before, we set $\nu_B$ to be the natural boundary measure on $\partial M_{4,r}^{+,\ell}$ and $\nu_I$ to be the marginal distribution of the uniform measure $\mu$ on $M_{4,r+\epsilon}^{+, \ell}$ in terms of $\theta$. The construction of analogous to $\mu_B$ and $\mu_I$ except that $\partial M_{4,r}^{+,\ell}$ consists of a single segment so there is no need to consider $\alpha$. In Figure 2, the height of the green curve corresponds the density $d \nu_B$, which is constant. The height of the red curve corresponds to the density $d \nu_I$. We can also define the associated projection $\pi: M_{4,r}^{+,\ell} \to \partial M_{4,r}^{+,\ell}$ which fixes the $\theta$ coordinate.
 
For fixed $\theta$, $d(v_2,v_4)$ is decreasing in $\ell$ and for fixed $\ell$, $d(v_2,v_4)$ is increasing in $\theta$. As such, $\nu_I$ is stochastically less than $\nu_B$. Furthermore, the transport  $\pi_\# \mu$ from $\mu$ to $\nu_I$ decreases the total curvature. This allows us to immediately repeat the previous argument involving Crofton's differential equation and prove Theorem 1, which we restate for convenience.
\begin{theorem}
For $1 \leq r \leq 2$,
$\bar \kappa_{4,r}$ is non-increasing in $r$.
\end{theorem}
 As before, this relies on the curvature monotonicity lemmas, but does not use the second lemma on stochastic ordering.

\section{Miscellaneous results}

For arbitrary $n$, it is possible to control the total curvature of $P_n$ when the diameter is either very large or very small. The following two lemmas can be obtained using straightforward estimates on angles between each edge, so we omit their proofs.
\begin{lemma}
If the diameter is 1, then the total curvature of an equilateral polygon with $n$ edges is at least $\frac{2 \pi}{3}n$.
\end{lemma}
\begin{lemma}
If $n$ is even and the diameter is greater than $n/2-\epsilon$ with $ \epsilon$ small, the total curvature of an equilateral polygon with $n$ edges is $2 \pi+ O(n \sqrt{\epsilon})$.
\end{lemma}

From the work in \cite{CS}, the expected total curvature is of an equilateral random polygon converges to $\frac{\pi}{2} n + \frac{\pi}{4}$ as $n$ gets large. The previous two inequalities give lower and upper bounds on $\kappa_{n,r, \mu B}$, respectively and in conjunction with Crofton's differential equation, these estimates show that $\bar \kappa_{n,r}$ is decreasing near $r=1$ for $n$ large and near $r=n/2$ for even $n$.

Furthermore, the final estimate can be applied to show a similar result for the knotting probability.  For $\epsilon$ sufficiently small, if an equilateral polygon has diameter at least $n/2-\epsilon$, then its total curvature is less than $4 \pi$. Appealing to the Fary-Milnor theorem, any such polygon must be unknotted. Therefore, the probability of knotting is decreasing when the confinement diameter is close to $n/2$.

\end{document}